\documentclass[preprint,12pt]{article}
\usepackage{graphicx,amssymb}
\usepackage{subfigure}
\usepackage{amsfonts,amsmath,amsthm}
\usepackage{multirow}
\usepackage{appendix}
\usepackage{microtype}

\newtheorem{thm}{Theorem}

\usepackage{comment}
\usepackage{mathrsfs} 
\usepackage{algpseudocode}
\usepackage{algorithm}
\usepackage{amsmath}
\usepackage[utf8]{inputenc}
\usepackage{dirtytalk}
\usepackage{tikz}
\usepackage{url}
\usepackage{color}
\usepackage{lipsum}

\newcommand\blfootnote[1]{%
  \begingroup
  \renewcommand\thefootnote{}\footnote{#1}%
  \addtocounter{footnote}{-1}%
  \endgroup
}
\usepackage[total={6.5in, 9in}, top=1in, left=1in]{geometry}
\usepackage{multirow}

\allowdisplaybreaks[4]

\title{Statistical Inversion Using Sparsity and Total Variation Prior And Monte Carlo Sampling Method For Diffuse Optical Tomography}

\author{Thilo Strauss$^a$, Sanwar Ahmad$^b$, Taufiquar Khan$^c$
\\
$^a$University of Washington, Seattle, School of Medicine \\ $^b$Colorado State University, Department of  Mathematics \\ $^c$Clemson University, School of Mathematical and Statistical Sciences
}
\date{}

\begin{document}
\maketitle
\begin{abstract}
In this paper, we formulate the reconstruction problem in diffuse optical tomography (DOT) in a statistical setting for determining the optical parameters, scattering and absorption, from boundary photon density measurements. A special kind of adaptive Metropolis algorithm for the reconstruction procedure using sparsity and total variation prior is presented. Finally, a simulation study of this technique with different regularization functions and its comparison to the deterministic Iteratively Regularized Gauss Newton method shows the effectiveness and stability of the method.
\blfootnote{$^a$Work done at the University of Washington, Seattle, WA.}
\blfootnote{$^c$Corresponding author: khan@clemson.edu }
\end{abstract}

%%%%%%%%%%%%%%%%%%%%%%%%%%%%%%%
\section{Introduction}
\label{sec:intro}

Diffuse Optical Tomography (DOT) is an imaging modality for probing highly scattering media by using low-energy visible (wavelength from 380nm to 750 nm) or near-infrared light (wavelength from 700 to 1200 nm). Light penetrates the media and interacts with its tissue. The predominant effects are absorption and scattering \cite{Gibson:review,Chance:1993,Delpy:1997,Hebden:1997}.  The widely accepted photon transport model is the Radiative Transfer Equation (RTE) \cite{Chandrasekhar:1960,Ishimaru:1978}, an integro-differential equation for the radiance, involving spatially varying diffusion and absorption parameters, which are a priori unknown. Hence the inverse problem consists of reconstructing an image of the optical properties (absorption and diffusion coefficients) of the tissue from measurements of some function of the photon density on the boundary.  The optical properties will vary significantly between background tissue and potential tumors, making DOT an attractive imaging technique.  Further, the scattering and absorption parameters are functional information, which are not acquired by standard x-ray attenuation type techniques, that can give information such as hemoglobin, water content, and lipid concentration \cite{Boas:2001}.\\
In practice, a low order diffusion approximation to the RTE is often adopted which is the standard model in DOT. The approximation is a parabolic and an elliptic differential equation in the time-dependent case and in the steady-state case or the frequency domain, respectively \cite{Arridge:1999,ArridgeSchotland:2009}. Most existing computational methods for the forward problems as well as inverse problems of photon migration in biological tissues are based on the approximation because of its simplicity compared to the full blown radiative transfer equation \cite{Chance:1995,Hielscher:1998}. It is well known that the DOT inverse problem is exponentially ill-posed \cite{Arridge:1999,Natterer:2001}. Image reconstruction in DOT is highly nonlinear and unstable \cite{Arridge:review,Khan:acm}. In fact, DOT is an excellent example of a severely ill-posed inverse problem and is getting more and more attention \cite{Arridge:review,Gibson:review,Khan:acm}.  The DOT inverse problem also includes electrical impedance tomography (EIT) as a special case \cite{borcea2002electrical}. There is great interest in understanding the inverse problem in DOT due to the huge impact in potential applications such as medical imaging for example neo-natal brain imaging, detection of breast cancer, and osteoarthritis detection. DOT is a preferred modality because of its low cost, non-invasiveness, and safe use of mainly near infrared (compared to x-ray) light radiation \cite{Gibson:review,Jiang:review}. In order to solve an inverse problem using the elliptic partial differential equation (PDE) model such as in DOT, many computational issues must be considered such as model error, nonlinearity, linearization via adjoint, uniqueness/non-uniquness, ill-posedness, regularization etc. (see \cite{Khan:acm,Natterer:review,Santosa:algorithm,Walker:backprojection, Kelley:iterativeoptimization,Koza:genetic,Boyd:convexoptimization,Tropp:greedy,hofmann2007convergence,ito2011regularization,
jin2012iterative,kaltenbacher2010convergence,resmerita2005regularization,ito2011new,ito2011newapproach,kaltenbacher2011adaptive,
schuster2012regularization} for details).\\
In order to overcome the main difficulties of the ill-posed inverse problem in DOT, regularization is required to find a reasonable solution for this problem.
There exists vast literature on how to regularize a nonlinear ill-posed problem such as DOT. For example, Tikhonov and iterated soft-shrinkage regularization methods for nonlinear inverse medium scattering problems \cite{lechleiter2013tikhonov}, a nonlinear ART or Kaczmarz method \cite{Natterer:review}, a conjugate gradient type method \cite{Pei:uniqueness}, variations/hybrid of Newton's method \cite{Natterer:review,natterer2002frechet}, diffusion-backdiffusion algorithms \cite{Arridge:modeling,Natterer:review}, total variation algorithms \cite{chung2005electrical} such as using level sets \cite{Dorn:adjoint,Dorn,Dorn2}, iteratively regularized Gauss-Newton \cite{Smirnova:ip} etc. have been investigated for DOT . However the reconstruction results are still not satisfactory for DOT. \\
In recent years, statistical methods and regularizations \cite{kaipio2000statistical, straussstatistical, Strauss2016Statistical, strauss2015statistical,sanwar2019} have been studied for reconstructing the EIT problem. In this paper we extend some of the results for EIT to DOT. That is, we formulate the DOT problem in terms of a posterior density, introduce statistical regularization methods and finally use the Markov Chain Monte Carlo to find a point estimate for the absorption and diffusion coefficients from boundary measurements. We note that classical algorithms such as the Metropolis Hastings \cite{metropolis1953equation, hastings1970monte, chib1995understanding} or the Gibbs Sample \cite{george1993variable} are not performing well for this particular problem in our experience. Hence, we are using a special kind adaptive Metropolis algorithm \cite{straussstatistical} to obtain a proper reconstruction of the absorption and diffusion coefficients in a suitable about of time. There is a vast amount of statistical literature on adaptive Metropolis algorithms such as  \cite{haario2001adaptive, gilks1998adaptive, gelfand1994markov, gelman1996efficient}.\\
This paper is organized in the following way: In section \ref{Model} the analytical forward problem of DOT is described. In section \ref{diff}, the Fr\'{e}chet differentiability of the DOT forward operator is discussed, which is used in section \ref{inverse} for the brief description of the iteratively regularized Gauss-Newton (IRGN) method used as comparison to the main result of this paper. Proof of the convergency of IRGN is presented in section \ref{irgn}. In section \ref{inversion}, the in statistical inverse problem is is formulated, mainly describing the posterior density, different regularization functions, the Markov Chain Monte Carlo method, and a pilot adaptive Metropolis algorithm. In section \ref{simulationStudy}, simulation from noisy measurements are presented to show the effectiveness and stability of the method and compare it to the classical Newton method. Finally in section \ref{conclusion}, we resume the paper presenting final concussions.
%%%%%%%%%%%%%%%%%%%%%%%%%%%%%%%%%%%%%%%%%
\section{The Forward Problem}\label{Model}
To model light traveling through tissue, the scattering and absorption qualities of the medium are required.  Specifically the reduced scattering coefficient $\mu_s'$ and absorption coefficient $\mu_a$ from the RTE, which models this type of physical system, are needed.  Approximating with the frequency domain diffusion approximation to the RTE, the result is a fairly simple model for optical tomography:
 \[ -\nabla \cdot \left( D \nabla u \right) + (\mu + ik) u = h \quad \mbox{ in } \Omega  \]
where $D=\frac{1}{3(\mu_a+\mu_s')}$ represents the diffusion coefficient, $\mu=\mu_a$ represents the absorption coefficient.  $u$ is the photon density, $h$ is the interior source, and $k=\omega/c$ for $\omega$ the frequency of the source and $c$ the velocity of light in the medium.  The solution $u$ represents the photon density.  The model was derived in detail in \cite{Arridge:1999}.  With this DOT model in hand, we have the Neumann
\begin{align*}
 -\nabla \cdot \left( D \nabla u \right) + (\mu + ik) u &= h, \quad \mbox{ in } \Omega\\
D \frac{\partial u}{\partial \nu}&=   f, \quad \mbox{on } \partial \Omega 
\end{align*}
forward problem.  The solutions to this problem is referred to as $F_N^{(k,q)}(h,g)$ where $q=(D,\mu)$ represents the parameters.  It is assumed that $f\in H^{-1/2}(\partial\Omega)$ and $g\in H^{1/2}(\partial\Omega)$.  Further, $\gamma_0:H^1(\Omega)\rightarrow H^{1/2}(\partial\Omega)$ denotes the Dirichlet trace map.  For the purposes of this manuscript, the homogeneous case is considered with $h=0$.

To simplify the solution space, it is assumed that $D,\mu\in L^\infty(\Omega)$, $0<D_0<D(x)<D_1$, and $0<\mu(x)<\mu_1<\infty$ for positive real constants $D_0, D_1$ and $\mu_1$.  Then the resulting photon density $u$ belongs to $H^1(\Omega)$.  The well-posedness of these problems has been shown in \cite{Dierkes:2002}. In \cite{natterer2002frechet}, the authors shows the uniqueness of the forward DOT operator. 
%%%%%%%%%%%%%%%%%%%%%%%%%%%%
\section{Fr\'{e}chet differentiability of the DOT operator}\label{diff}
The inverse problem of the DOT problem is to estimate the diffusion ($D$) and the absorption ($\mu$) coefficient for a source function $f$. We solve the inverse problem using iteratively regularized Gauss-Newton (IRGN) method, for which a Fr\'{e}chet differentiablity of the map $\gamma_0{F_N}^{(k,q^\dagger)}(0,f)$ is a necessary tool. Recall that Fr\'{e}chet differentiability of $\gamma_0{F_N}^{(k,q^\dagger)}(0,f)$ is defined as
\begin{align*}
    \lim_{||\eta||_\infty \to 0} & \frac{||\gamma_0{F_N}^{(k,q^\dagger)}(q+\eta:0,f)-\gamma_0{F_N}^{(k,q^\dagger)}(q:0,f)-\gamma_0{F'_N}^{(k,q^\dagger)}(q:0,f)\eta||}{||\eta||_\infty}\\ &= 0
\end{align*}
\begin{thm}\label{frechet_diff}
Suppose $q_1 = (D_1,\mu_1), q_2 = (D_2,\mu_2)$ are two pairs of real-valued $L_\infty(\Omega)$ functions satisfying,
\[ 0 < m_D \leq D_{1,2}(x) \leq M_D, \quad 0 \leq \mu_{1,2}(x) \leq M_\mu, \mbox{ for some } m_D, M_D, M_\mu >0 \]
then there exists a constant $C$, such that,
\begin{align} \label{frechet}
    ||\gamma_0{F_N}^{(k,q^\dagger)}(q_2:0,f)-\gamma_0{F_N}^{(k,q^\dagger)}(q_1:0,f)-&\gamma_0{F'_N}^{(k,q^\dagger)}(q_1:0,f)\eta|| \nonumber \\ &
    \leq C ||\eta||_\infty^2
\end{align}
where $\eta = q_2 - q_1$. In particular, $\gamma_0 {F'_N}^{(k,q^\dagger)}(q_1:0,f)$ is the Fr\'{e}chet derivative of $\gamma_0{F_N}^{(k,q^\dagger)}(0,f)$ with respect to $q$ at $q_1$.
\end{thm}
% \textbf{Theorem.}  \\
\begin{proof}
See \cite{natterer2002frechet}.
\end{proof} 
%%%%%%%%%%%%%%%%%%%%%%%%%%%%
\section{Inverse Problem}\label{inverse}
The Neumann-to-Dirichlet map $\gamma_0F_N^{(k,q^\dagger)}(0,f)$ amounts to conducting the forward problem experiment over the true parameters $q^\dagger$.  Because of imprecise measurements, the boundary information used is denoted
\[ g^\gamma = \gamma_0F_N^{(k,q^\dagger)}(0,f)+\xi\]
where $\xi$ represents the noise. 
The forward problem uses $q$ to find the boundary data associated with a given source $f$. The inverse problem is to estimate the parameter $q = (D,\mu)$, for a set of source functions. Since the forward problem is well-posed, it is sufficient to consider $q \in \{ L_\infty \times L_\infty: 0 < m_D \leq D \leq M_D, 0 < \mu < \infty \}$. However, in order to establish the uniqueness of the solution to the inverse problem, we consider $q \in \{H_2(\Omega) \times L_2(\Omega): 0 < m_D \leq D \leq M_D, 0 < \mu < \infty \}$.
\begin{thm}
A solution $q$ to the DOT inverse problem exists and is unique. That is given the measurements on the boundary, we can reconstruct a unique spatial map for the parameter $q$.
\end{thm}
\begin{proof}
The proof of this theorem is well-established in the literature, see \cite{sylvester1987global}.
\end{proof}
Since the inverse problem is severely ill-posed, in order to recover $q$, we minimize the following cost functional from the finite set of measurements. 
\begin{align*}
J_\gamma(q) = ||\gamma_0F_N^{(k,q)}(0,f) - g^\gamma||_2^2	
\end{align*}
where $g^\gamma$ approximates the exact data $g$ with the accuracy $\xi$, i.e.,
\begin{equation}\label{noise}
||g-g^\gamma|| < \xi.
\end{equation}
However, regularization is needed to improve the ill-posed problem and instead, we minimize,
\begin{align*}
J_\gamma(\alpha,q) = ||\gamma_0F_N^{(k,q)}(0,f) - g^\gamma||_2^2 + \alpha R(q-q^*)
\end{align*}
where $\alpha$ is the regularization parameter, $R(\cdot)$ is the regularization term and $q^*$ is the known background. \\
The regularization function is represented by a norm for most analytical methods. 
In this paper we used one of the most successful methods for solving the ill-conditioned problem, $R_{\ell_2}$ known as Tikhonov regularization. The cost functional from Tikhonov regularization is
\begin{equation}\label{cost}
J_\gamma(\alpha,q) = \frac{1}{2}||\gamma_0F_N^{(k,q)}(0,f) - g^\gamma||_2^2+ \frac{\alpha}{2}||\nabla q||_2^2 %||W(q-q^*)||_2^2,
\end{equation}%where $W$ is weight matrix. 
There are several iterative approaches to minimize (\ref{cost}). In this paper, we used a modified iteratively regularized Gauss-Newton (IRGN) method for the minimization.\\
Suppose $\alpha_k$ is some sequence of regularizing parameters satisfying the conditions
\begin{equation}\label{lambda}
\alpha_k \geq \alpha_{k+1} >0, \quad \sup_{k \in \mathbb{N}\cup \{0\} } \frac{\alpha_k}{\alpha_{k+1}} = \hat{d} < \infty, \quad \lim_{k \to \infty} \alpha_k = 0.
\end{equation}
Let the unique global minimum of (\ref{cost}) be denoted by $\tilde{q}$. Assume $\tilde{\gamma}$ satisfies the invertibility conditions and the conditions for $F$ as described in \cite{Smirnova:ip}. Then, the unique global minimum of (\ref{cost}) is explicitly given by
\begin{align}
\tilde{q} = &q_{k} - (\gamma_0 F'(q_k)^T \gamma_0 F'(q_k) + \alpha_k \triangle_k q)^{-1} \nonumber\\
&\{\gamma_0 F'(q_k)^T (\gamma_0 F(q_k) - g^\gamma)+\alpha_k \triangle_k q \} 
%W_2 (q_k-q^*)\},
\end{align}
where $q_k$ is the $k$-th approximation to $q$ and $\gamma_0 F'(q_k)$ is the Jacobian matrix at the $k$-th iteration. The term $\triangle_k$ is the discrete Laplacian operator.%, and $W_2 = W^T W$. 
The above algorithm is generalized further using a line search procedure, as discussed in \cite{Smirnova:ip}, by introducing a variable step size, $s_k$, such that
\begin{equation}\label{step}
0 < s < s_k \leq 1.
\end{equation}
The modified IRGN algorithm is then 
\begin{align}\label{irgn}
q_{k+1} = & q_{k} - s_k(\gamma_0 F'(q_k)^T \gamma_0 F'(q_k) + \alpha_k \triangle_k q_k)^{-1} \nonumber \\ 
& \{\gamma_0 F'(q_k)^T (\gamma_0 F(q_k) - g^\gamma)+\alpha_k \triangle_k q_k)\}.
\end{align}
Due to the inexact nature of $g^\gamma$, we adopt a stopping rule from \cite{bakushinsky2005application} to terminate the iterations at the first index $\mathcal{K} = \mathcal{K}(\delta)$, such that
\begin{equation}\label{stop}
||\gamma_0 F(q_\mathcal{K}) - g^\gamma||^2\leq \rho \xi < ||\gamma_0 F(q_k) - g^\gamma||^2, \quad 0 \leq k \leq \mathcal{K}, \quad \rho > 1.
\end{equation}
The line search parameter $s_k$ is chosen to minimize the scalar objective function
\begin{equation}\label{stsearch}
\Phi(s) = J(q_k+s p_k)
\end{equation}
where $p_k$ is the search direction, which solves
\begin{equation}\label{dir}
(\gamma_0 F'(q_k)^T \gamma_0 F'(q_k) + \alpha_k \triangle_k q_k)p_k = -\gamma_0 F'(q_k)^T (\gamma_0 F(q_k) - g^\gamma)+\alpha_k \triangle_k q_k)
\end{equation}
This step is accomplished through a backtracking strategy until either one of the strong Wolfe conditions,
\begin{eqnarray}
J(q_k+s p_k) \leq J(q_k) + \gamma_1 s \nabla J(q_k)^T p_k \label{wolf1}\\
|\nabla J(q_k+s p_k)^T p_k| \leq |\gamma_2 \nabla J(q_k)^T p_k|. \label{wolf2}
\end{eqnarray}
is satisfied \cite{nocedal1999numerical}, or the maximum number of backtracking steps has been reached. We use the theoretically derived values of $\gamma_1 = 0.0001$ and $\gamma_2 = 0.9$, derived in \cite{nocedal1999numerical}. The basic convergence result for the IRGN algorithm (\ref{irgn}) combined with the stopping rule (\ref{stop}) is established by the theorem provided in the next section.
%%%%%%%%%%%%%%%%%%%%%%%%%%
\section{Convergence of IRGN}\label{irgn}
Assume that $F$ is a nonlinear operator acts on the Hilbert spaces $(H,H_1)$, $F:D(F) \subset H \to H_1 $, and $F$ is Fr\'{e}chet differentiable in $D(F)$. We consider minimizing the functional 
\begin{equation}
    J(q) = ||F(q)-g_\delta||^2_{H_1}
\end{equation}
where $g_\delta$ approximates the exact data $g$ with the accuracy $\delta$, i.e.,
\begin{equation}
    ||g-g_\delta|| \leq \delta \label{noise}
\end{equation}
Our interest is to find an element $\hat{q} \in D(F)$, s.t.
\begin{equation}
    ||F(\hat{q})-g||_{H_1} = \inf_{q \in D(F)} ||F(q)-g||_{H_1} = 0 \label{existence}
\end{equation}
Consider the following conditions hold
\begin{align}
    ||F'(q_1)|| &\leq M_1, \quad \mbox{ for any } q_1 \in B_\eta(\hat{q}) \label{M1} \\
    ||F'(q_1)-F'(q_2)|| &\leq M_2 ||q_1-q_2||, \mbox{ for any } q_1,q_2 \in B_\eta(\hat{q}) \label{M2}
\end{align}
where $B_\eta(\hat{q}) = \{ q \in H : ||q-\hat{q} || \leq \eta \} \subset D(F)$.
The convergence analysis of IRGN is done under the source condition 
\begin{align}
    L^*L(\hat{q} - q) = F'^*(\hat{q}) S, \quad S = \{ v \in H: ||v|| \leq \varepsilon \} \label{source}
\end{align}
and by the following theorem.
\begin{thm}\label{conv_thm}
Assume that\\
(1) $F$ satisfies (\ref{M1}) and (\ref{M2}) with $\eta = l \sqrt{\tau_0}$, conditions (\ref{noise}) and (\ref{existence}) holds.\\
(2) The regularization sequence $\{\alpha_k\}$ and the step size sequence $\{s_k\}$ are chosen according to (\ref{lambda}) and (\ref{step}), respectively. \\
(3) Source condition (\ref{source}) is satisfied. \\
(4) The linear operator $L^*L$ is surjective and there is a constant $m>0$ such that 
\begin{align}
    \langle L^*Lh,h \rangle \geq m ||h||^2, \quad \mbox{ for any } h \in H \label{LL}
\end{align}
(5) Constants defining $F$ and the iteration are constrained by 
\begin{align}
    \frac{M_2 \varepsilon}{m}+\frac{d-1}{d\alpha}+ \sqrt{\frac{\varepsilon}{m} \left( \frac{M_2}{2}+\frac{M_1^2}{(\sqrt{\rho}-1)^2} \right) } \leq 1 \label{Mbound}\\
    \frac{||q_0-\hat{q}||}{\sqrt{\alpha_0}} \leq \frac{\varepsilon}{\sqrt{m}\left( 1-\frac{M_2 \varepsilon}{m}-\frac{d-1}{d\alpha}\right)} = l \label{lbound}
\end{align}
Then,\\
(1) For iterations (\ref{irgn})
\begin{align}
    \frac{||q_k-\hat{q}||}{\sqrt{\alpha_k}}\leq l, \quad k = 0,1,..., \mathcal{K}(\delta) \label{iterbound}   
\end{align}
(2) The sequence $\{\mathcal{K}(\delta) \}$ is admissible, i.e. 
\begin{align}
    \lim_{\delta \to 0}||q_{\mathcal{K}(\delta)}-z|| = 0, \label{converge}
\end{align}
$z$ is $\arg \inf_{q\in D(F)}||F(q)-g||_{H_1}$.
\end{thm}
\begin{proof}
See \cite{Smirnova:ip}, for the details of the proof.
\end{proof}
%%%%%%%%%%%%%%%%%%%%%%%%%%%%%%%
\section{Statistical Inverse Problem}\label{inversion}
Modeling the forward problem uses knowledge of the absorption, $\mu_a$, and the scattering, $\mu_s'$ parameters to find the boundary data associated with a given source.   The inverse problem instead uses knowledge of the source and boundary data, and seeks to uncover the absorption and scattering parameters.  Our stated goal is to recover $\mu_a$, and $\mu_s'$. However, considering the proposed model for optical tomography in this document, this is equivalent to the recovery of $q=(D,\mu)$.  These parameters have a high contrast between a tumor and the background and consequently acquiring an accurate representation of these parameters will allow for the detection of homogeneity's.  

Because, we can not directly calculate the inverse to the forward problem, the inverse problem can be represented as a minimization problem of a cost functional. One possibility is using the standard least squares fitting as discrepancy functional,
\[ J(q)=\frac{1}{2}\left\|\gamma_0F_N^{(k,q)}(0,f)-g^\gamma\right\|^2_{L_2(\partial\Omega)}\]
for a source $f$ and measurement $g^\gamma$. However, minimizing $J(q)$ won't result in a proper reconstruction due to the high ill-posedness of the problem. Hence, typical analytical approaches try to find an estimate $\hat{q}$ of $q$ by solving
\[\hat{q}=argmin_{q}J(q) + \alpha R(q)\]
where $R(q)$ is a regularization function and $\alpha>0$ the regularization parameter. In this section we shall reformulate this minimization problem into finding a Bayes' estimate from a posterior density.

\subsection{The Posterior Density}

With a slight abuse of notation we denote $D,\mu,q, g$ as its continuous setting as well as its discrete analogue. Further, let $\epsilon$ be the discrete measurement noise of the measurements $g$. Let $q^*, D^*, \mu^*, \epsilon^* $ and $g^*$ be the random variables for the deterministic values $q, D, \mu, \epsilon$ and $g$ respectively. Let $D, \mu$ be $n\in\mathbb{N}$ dimensional parameters. Hence, $q$ is a $2n$ dimensional parameter to be recovered. The measurements $g$ and the measurement noise $\epsilon$ are assumed to be both $m\in\mathbb{N}$ dimensional parameters. Let $\Theta(q,\epsilon)$ denote the discrete operator that solves the forward problem on $\partial\Omega$. That is, if the forward problem is well-posed, it follows that $g=\Theta(q,\epsilon)$.

The solution of a statistical inverse problem of the parameter $q$ is generally defined as the posterior density of $q^*$ given the measurements $g$ and the measurement noise $\epsilon$, e.i. $\pi_{q^*}(q|g,\epsilon)$.  The posterior density  $\pi_{q^*}(q|g,\epsilon)$ shall be derived first. From the statistical point of view it is clear that
 \[\pi_{g^*}(g|q,\epsilon)=\delta(g-\Theta(q,\epsilon)),\]
 where $\delta$ is the Dirac delta function. The Dirac delta function is not a function in a strict sense, however the integral over $\delta(\cdot)$ is in fact a distribution function. The joint density of $q^*, g^*$ and $\epsilon^*$ is
 \[\pi_{q^*, g^*, \epsilon^*}(q, g, \epsilon)=\pi_{g^*}(g| q, \epsilon)\pi_{q^*, \epsilon^*}(q, \epsilon)=\delta(g-\Theta(q,\epsilon))\pi_{q^*, \epsilon^*}(q, \epsilon).\]
However, the measurement noise $\epsilon$ is a unobservable quantity and should be integrated out of the model. The joint density of $q^*$ and $ g^*$ is obtained by
 \begin{equation}\pi_{q^*, g^*}(q, g)=\int_{\mathbb{R}^m}\pi_{q^*, g^*, \epsilon^*}(q, g, \epsilon)d\epsilon=\int_{\mathbb{R}^m}\delta(g-\Theta(q,\epsilon))\pi_{q^*, \epsilon^*}(q, \epsilon)d\epsilon.\label{simpThisEq}\end{equation}
In order to solve this integral more assumptions on the measurement noise $\epsilon^*$ are required. Assume that the measurement noise $\epsilon^*$ is independent from the variable of interest $q^*$, i.e. $\pi_{q^*, \epsilon^*}(q, \epsilon)=\pi_{q^*}(q)\pi_{\epsilon^*}(\epsilon)$, and that the measurement noise is additive, i.e. $g=\Theta(q)+\epsilon$. Now equation (\ref{simpThisEq}) can be simplified to
 \[\pi_{q^*, g^*}(q, g)=\int_{\mathbb{R}^m}\delta(g-\Theta(q)-\epsilon)\pi_{q^*}(q)\pi_{\epsilon^*}(\epsilon)d\epsilon=\pi_{q^*}(q)\pi_{\epsilon^*}(g-\Theta(q)).\]
It follows that the posterior density of $q^*$ given the measurements $g$ is 
\begin{equation}\label{posteriorDist}
\pi_{q^*}(q| g)=\frac{\pi_{q^*, g^*}(q, g)}{\pi_{g^*}(g)}= \frac{\pi_{q^*}(q)\pi_{\epsilon^*}(g-\Theta(q))}{\int_{\mathbb{R}^{2n}} \pi_{q^*}(q)\pi_{\epsilon^*}(g-\Theta(q))dq}\propto \pi_{q^*}(q)\pi_{\epsilon^*}(g-\Theta(q)), 
\end{equation}
where $\int_{\mathbb{R}^{2n}} \pi_{q^*}(q)\pi_{\epsilon^*}(g-\Theta(q))dq>0$ is a unknown constant. To obtain a proper defined posterior density the densities of $q^*$ and $\epsilon^*$ must be specified. Assume that the measurement noise is Gaussian distributed, i.e. 
\[\pi_{\epsilon^*}(g-\Theta(q))\propto\exp\left[-\frac{1}{2}(g-\Theta(q))'C^{-1}(g-\Theta(q))\right].\]
Further, assume that the prior density, the density of $q^*$, is 
\[\pi_{q^*}(q)\propto \chi_A(q)\exp[\alpha R(q)],\]
where $R(\cdot)$ is a regularization function, $\alpha>0$ is a constant and $\chi_A(q)$ is a indicator function with $A=(0,\infty)^{2n}$. 

\subsection{Regularizing Functions}

In DOT regularization functions are required to obtain proper reconstructions despite the high ill-posed of the problem. In this section several choices for the regularization function $R(\cdot)$ %which are used in the analytical and  the statistical setting 
shall be introduced.
However, in the statistical setting somewhat more general versions of the classical analytical choices can be obtained.

In DOT regularization function for $D$ and $\mu$ must be chosen simultaneously. In our experience proper regularized reconstructions for $D$ and $\mu$ are obtained is by letting 
 \begin{equation}\label{DOTregExample}
 R(q)=R(D,  \mu)=\beta_1R\left(\frac{\mu}{\mu^b}\right)+\beta_2R\left(\frac{\mu_s'}{\mu_s^b}\right),
  \end{equation}
 where $R$ is a regularization functions, $\beta_1, \beta_2>0$ with $\beta_1+ \beta_2=1$ are constants, $\mu^b$, and $\mu_s^b$ are the expected backgrounds of  $\mu$, and $\mu_s'$ respectively, where 
 \[\mu_s'=\frac{1}{3D}-\mu.\]
The reason for dividing $\mu$, and $\mu_s'$, by their typical backgrounds is that it simplifies the choice of $\beta_1$, and $\beta_2$. That is, setting $\beta_1=\beta_2=\frac{1}{2}$ could be a reasonable choice, if an approximately liner relationship between $\mu$, and $\mu_s'$ can be assumed.
%%%%%%%%%%%%%%%%%%%%%%%%%%%%%%%%%%
\subsubsection{The $\ell_p$ Regularizations} \label{lpreg}
The $\ell_p$ regularization $R_{\ell_p}(y)$ is defined as
\[ R_{\ell_p}(y):=\sum_{i=1}^n c_i|y_i-y^b_i|^p, \]
where $c_i$ represent weights, $0<p\leq 2$ is a constant and $y^b$ the typical background from $y$. In case that the regularization of (\ref{DOTregExample}) is chosen, it follows that $y^b=1$. Note that $R_{\ell_p}(y)$ is a norm if $p\geq1$ and would only define a metric in case $0<p<1$. For analytical methods it is usually necessary that the regularization function  represents a norm, while for statistical reconstruction the case when $0<p<1$ can potentially also be handled. 

%In theory a good choice for the weights would be large values at the boundary and exponentially decreasing values towards the center of $\Omega$. This is because the variance  is smaller on the boundary than in the center \cite{palamodov2002gabor}. The $\ell_p$ regularization enforces sparsity when $0< p \leq 1$  and enforces smoothness when $p\geq 2$. 

\subsubsection{The Total Variation Regularization} \label{totalVar}

%The idea behind the total variation regularization is to obtain smooth images. This is meaningful in most practical applications. 

The total variation regularization is defined as 
\[R_{TV_c}(y_c):=\int_{\Omega}|\nabla y_c|dx, \]
where $y_c$ the continuous version of the parameter of interest  $y$. The discrete analogue for a two-dimensional body of the total variation regularization $R_{TV_c}$ \cite{kaipio2000statistical,burgel2017sparsity} is 
\[R_{TV}(y):=\sum_{i=1}^hl_i|\triangle_iy|,\]
where $l_i$ is defined as the length of the edge corresponding to the $i^{th}$ adjacent pixel and 
\[\triangle_i=(0, 0, ..., 0, 1_{a_{(1,i)}}, 0, ..., 0, -1_{a_{(2,i)}}, 0, ..., 0),\]
with $a=(a_{(j,i)})_{i=1, \text{ } j\in\{1,2\}}^h$ is the set containing the numbers of all adjacent pixel tuples $(a_{(1,i)}, a_{(2,i)})$.

\subsubsection{General Regularizations} \label{GRegul}

In general, there is no need to be restricted to the typical analytical choices, we can rather choose $\pi_{y^*}(y)$ to be any kind of prior density on $\Omega$. However, in order to obtain proper reconstructions it is recommended to choose from meaningful continuous regularization functions of $y$ which are integrable over $\Omega$. 
 
One prior that in our experience gives good results is considering both the $\ell_p$ and total variation prior,
\begin{equation}
R_{G}(\zeta)=\alpha_1R_{\ell_p}(y)+\alpha_2R_{TV}(y), \label{genReg}
\end{equation}
where $\alpha_1, \alpha_2>0$ are constants. This regularization function, with proper choice of $\alpha_1, \alpha_2>0$ usually leads to images which are have a smooth background image while detecting the anomalies accurately. 

 \subsection{The Markov Chain Monte Carlo Method} 

The posterior density with several meaningful prior densities (regularization's) has been introduced.  However, the abstract formulation of the posterior density tells little about the absorption coefficient $\mu$ and the diffusion coefficient $D$. That is, a estimator for $q^*=(\mu^*, D^*)$ given the measurements $g^*$ must be found, e.i. the Bayesian estimator
\[E(q^*| g)=\int_{\mathbb{R}^n}q\pi_{q^*}(q| g)dq.\]
Given that the posterior density $\pi_{q^*}(q| g)$ does not have a closed form, there is no direct method of finding the Bayesian estimate $E(q^*| g)$. However, the Markov Chain Monte Carlo Method (MCMC) can be used to generate a large random sample $\{q^{(i)}\}_{i=B+1}^N$ from the posterior density $\pi_{q^*}(q| g)$ in order to approximate the Bayesian estimate by its sample mean,
\begin{equation}
E(q^*|g)=\int_{\mathbb{R}^n}q\pi_{q^*}(q|g)dq\approx\frac{1}{N-B}\sum_{i=B+1}^Nq^{(i)},
\end{equation}
where $N$ is the total number of samples and $B$ is the burn in time. The burn in time $B$ is the number of random samples which can still not be considered as real random samples from the posterior density. Typical algorithms to generate such a random samples from a posterior density are the Gibbs sampler or the Metropolis-Hastings algorithm. The Metropolis-Hastings algorithm seems to be computationally less expensive than the Gibbs sampler for the considered problems. Here the Metropolis-Hastings algorithm shall be briefly described, for a more detailed description see \cite{chib1995understanding, strauss2015statistical}.

\subsection{The Metropolis-Hastings Algorithm} 

Consider a Markov chain on the continuous state space $(E, \mathscr{B}, \mathscr{M})$ where $\mathscr{B}$ is a Borel $\sigma$-algebra on $E$ and $\mathscr{M}$ the normalized Lebesgue measure on $E$. Let $E\subseteq  \mathbb{R}^d$ be the support from the target distribution. This means $E$ is the set containing all values a state $x^{(i)}$, of the chain $\{x^{(i)}\}$, can take. Furthermore, let $P(x;A)$ denote a  transition kernel for $x\in E$, where $A\in \mathscr{B}$. A transition kernel $P(x;A)$ represents the probability of jumping from a current state $x$ to another state in the set $A$. It is desirable to find a transition kernel $P(\cdot;\cdot)$ such that it converges to an invariant distribution $\pi^*$. Here $\pi^*$ represents the distribution of the posterior density $\pi$, e.i. the density defined in (\ref{posteriorDist}) when considering the DOT problem. After some analysis it can be shown that the transition kernel 
\[P_{MH}(x;A):=\int_A p(x,y)\alpha(x,y)dy+\left[1-\int_{\mathbb{R}^d}p(x,y)\alpha(x,y)dy \right]\chi_A(x),\]
converges to the invariant distribution $\pi^*$. Here $\chi_A$ is the indicator function over the set $A$, and $p(x, y)$ is a proposal density, that is, a density which generates a new candidate random sample $y$ from a current random sample $x$. For example, $p(x,\cdot)$ could be a multivariate normal density with mean $x$. The acceptance ratio 
\begin{align*}
    \alpha(x,y)= \left\{
\begin{array}{c l}      
    \min \left[\frac{\pi(y)p(y,x)}{\pi(x)p(x,y)}, 1 \right], & \mbox{if } \pi(x)p(x,y)>0, \\
    1, &  \mbox{otherwise},
\end{array}\right.
\end{align*} 
is the probability of accepting a new random sample $y$. Note that the formulation of the acceptance can be simplified if the proposal density is symmetric, i.e. $p(x,y)=p(y,x)$ $\forall x,y\in E$. 

The choice of a proper proposal density is vital for the Metropolis algorithm to obtain real random samples from the target distribution $\pi^*$ in a suitable amount of time. This choice is usually very complicated because the target density is generally unknown \cite{gelman1996efficient, haario2001adaptive, gilks1998adaptive}. One method of eliminating this problem is by using an adaptive Metropolis algorithm. Hence, iteratively adapting the proposal density based on previous samples from the chain. However, adaptive Metropolis algorithms usually produce a non-Markovian kind of chain, i.e. $P(X^{(n)}|X^{(0)}, X^{(1)}, ..., X^{(n-1)})\neq P(X^{(n)}|X^{(n-1)})$. This would require to establish the correct ergodic properties. 

The adaptive Metropolis algorithm, proposed in \cite{gilks1998adaptive}, changes the covariance matrix of the proposal distribution at atomic times in order to preserve the Markov property of the chain. An atomic time on a continuous state space  $(E, \mathscr{B}, \mathscr{M})$ is a set $A\in \mathscr{B}$ with $\pi^*(A)>0$ such that if $X^{(n)}\in A$ the chain $X^{(n+1)}, X^{(n+2)}, ...$ is independent of $X^{(0)}, X^{(1)}, ..., X^{(n)}$.  Even though this method appears to be very attractive it is practically very complicated to find proper atomic times for high dimensional problems  \cite{gilks1998adaptive}.  Another approach, introduced in \cite{haario2001adaptive}, is to adapt the covariance matrix of a normal proposal distribution in every iteration after an initial time $t_0$ in such a way that the correct ergodic properties of the chain can be established  even if the chain is itself non-Markovian. The approach used in this paper  is to adapt the proposal distribution a finite amount of times and begin the burn-in time after the last adaption \cite{straussstatistical, gelfand1994markov, strauss2015statistical, strauss2015statisticalDarcy}. This method can not guarantee that a optimal proposal distribution for the target distribution is obtained after the last adaption. However, the convergence speed respect to the classical Metropolis-Hastings algorithm is usually increased considerable while still maintaining all good properties of the chain after the pilot time. That is, the pilot adaptive Metropolis algorithm (Algorithm \ref{PilotA}) still generates a Markov chain after the pilot time.

\subsection{A Pilot Adaptive Metropolis Algorithm}

The idea of this algorithm is to update the proposal distribution by changing its covariance matrix in such a way that the acceptance ratio of the chain after the last adaption is close by the optimal acceptance ratio $a_o$ of the chain.  There is no analytical framework for the choice of such a optimal acceptance ratio $a_o$ when the target distribution is unknown. Hence, the choice of $a_o$ is usually based on the result from \cite{gelman1996efficient} where they show that for a normal target and proposal distribution  the optimal acceptance ratio is approximately $.45$ in the one dimensional case and $.234$ when the number of dimension converges to infinite. 

Suppose we wish to perform $M$ adaptions, one every $m$ iterations, where $1<mM<B<N$. Let $c_i$ denote a variable recording whether or not the $i$-th iteration of the considered algorithm has been accepted, 
\begin{align} 
c_i:=\left\{\begin{array}{c l}      
   1, & \mbox{$i$-th iteration has been accepted,}\\
    0, & \mbox{else.}
\end{array}\right.
\end{align}
The estimator for the acceptance ratio for the $j$-th proposal distribution is \\ $\bar{a}_j=\frac{1}{m}\sum_{i=(j-1)m+1}^{jm}c_i$. Let $1\gg\epsilon>0$, where $100\epsilon$ is the percentage of change per adaption in the entries of the covariance matrix $C$ of the proposal distribution. In other words, the  $j$-th adaption modifies the current covariance matrix $C_{j-1}$ in the following way,
\begin{equation}
C_{j}=\Xi_{PAM}(C_{j-1}):=\left\{\begin{array}{c l}      
   (1+\epsilon)C_{j-1}, & \mbox{if }  \bar{a}_j>a_o,\\
   C_{j-1}, & \mbox{ if }  \bar{a}_j=a_o,\\
    (1-\epsilon)C_{j-1}, & \mbox{ if }  \bar{a}_j<a_o.
\end{array}\right.
\end{equation}
Informally speaking, the algorithm modifies the covariance matrix in the pilot time $mM$ in such a way that it comes closer to one which has an optimal acceptance ratio. Then the standard Metropolis-Hastings algorithm begins with the latest state and proposal distribution of the pilot time. In Algorithm \ref{PilotA} the pilot adaptive Metropolis algorithm is recapitulated with an arbitrary starting state $x^{(0)}\in E$ and a starting guess for the positive definite covariance matrix $C_0$. 

\begin{algorithm}
\caption{A Pilot Adaptive Metropolis Algorithm.}
\label{PilotA}
\begin{algorithmic}
 \State $j=0;$
    \For{i = 1 to N}
    \If{$i\equiv 0 \bmod{m}$ and $i\leq mM$} 	
    	\State $C_{j}=\Xi_{PAM}(C_{j-1})$;
	\State j++;
    \EndIf
     \State Generate $y$ from $q_{C_j}(x^{(i-1)},\cdot)$ and $u$ from $\mathcal{U}(0,1)$;
     \If{$u \leq \alpha(x^{(i-1)},y)$}
     	\State $x^{(i)}=y;$
	\Else
	\State $x^{(i)}=x^{(i-1)};$
	\EndIf
     \EndFor\\
    \State  \Return{$\left\{ x^{(1)}, x^{(2)}, 	\ldots, x^{(N)}\right\}$}
  \end{algorithmic}
\end{algorithm}

The convergence from the pilot adaptive algorithm (Algorithm \ref{PilotA}) is described in \cite{strauss2015statistical}. It follows from Theorem \ref{convPilot}.
\begin{thm} \label{convPilot}
Let $\pi$ be the density of the target distribution supported on a bounded measurable subset $E \subset \mathbb{R}$. Further, assume that $\pi$ is bounded from  above.  Let $\epsilon > 0 $ and let $\mu_0$ be any initial distribution on $E$. Define AM-chain ${x_n}$ by the generalized transition probabilities \cite{haario2001adaptive}. Then the AM-chain simulates properly from the target distribution $\pi^*$ , that is, for any bounded and measurable function $f: E \rightarrow \mathbb{R}$, the equality
$lim_{n\rightarrow\infty}\frac{1}{n+1}(f(x^{(0)})+f(x^{(1)})+ ... +f(x^{(n)})) \approx \int_{E}f(x)\pi^*(dx,)$
holds almost surly.
\end{thm}
\begin{proof} 
See \cite{haario2001adaptive, strauss2015statistical}.
\end{proof}
It follows that the Algorithm \ref{PilotA} converges after the the pilot time as in Theorem \ref{convPilot}. It is important to note that when applying Algorithm \ref{PilotA}  the chain only satisfies the Markov property after the last adaption at time $mM$. However, the chain usually still moves towards the high probability areas of the target distribution during the pilot time \cite{straussstatistical, strauss2015statistical, Strauss2016Statistical}. 

One problem left to answer is how to chose the burn in time $B$ and the number of total samples $N$ in order obtain proper samples from the posterior density. Even if, to our knowledge, there is no framework to choose $B$ and $N$ prior to run the algorithm, there are a few convergence diagnostics, such as Gelman and Rubin \cite{gelman1992inference, brooks1998general}, Geweke \cite{el2006comparison}, Raftery and Lewis \cite{raftery1996implementing} or empirical methods like trace plots or auto correlation plots. All this convergence diagnostics can be applied to test whether or not a chain converged. However, each may obtain misleading results in some cases, hence it is recommended to use multiple of the above diagnostics simultaneously.

\section{Simulation Study}\label{simulationStudy}

In this section, reconstructions using a pilot adaptive Metropolis algorithm (Algorithm \ref{PilotA}) are presented.   The photon density measurements were simulated on a mesh of 2097 triangles (Figure \ref{meshes}.A), then $1\%$ Gaussian measurement noise has been added to the measurements. The reconstructions where made, based on the noisy measurement on a mesh of 541 triangles (Figure \ref{meshes}.B). The mesh for the simulations and the reconstructions were chosen to be different in order to avoid committing an inverse crime. We assume that the parameters of interest are known and constant on the boundary, hence the number of parameters to be estimated for $D$ and $\mu$ where with 477 parameters somewhat smaller then the number of triangles. Note that in order to reduce the computational time the starting guess $x^{(0)}\in E$ has been selected to be the reconstruction after a few iterations of an analytical minimization algorithm. In all reconstructions we choose to perform $M=600$ adaptions, one every $m=50$ iterations, e.i. the pilot time was chosen to be $mM=30.000$. Further, the burn in time was chosen to be $B=100.000$ and the total number of samples was $N=150.000$. We compared this method with the Newton method which we run until convergence at approximately 150 iterations.

\begin{figure}
\centering
\includegraphics[scale=0.35]{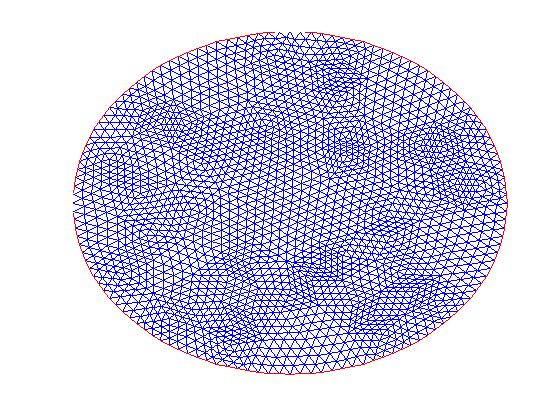}\includegraphics[scale=0.35]{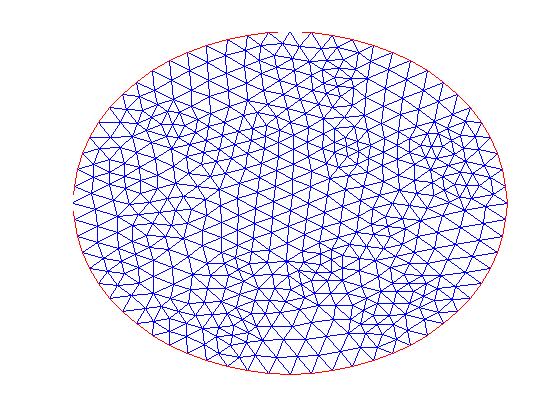}{\\\vspace*{-.5cm}\hspace*{.4cm}\tiny A\hspace*{6.9cm}\tiny B}
\caption{A: Mesh containing 2097 triangles used for simulating measurements. B: Mesh containing 541 triangles used for the parameter reconstruction.}\label{meshes}
\end{figure}

In figures \ref{incone}-\ref{inc_1c}, image A and B represent the true parameters $\mu$ and $D$ in mesh \ref{meshes}.A, images C-D and E-F are the corresponding reconstructions using the TV reconstruction (Section \ref{totalVar}) and a mixture of the TV and the $\ell_1$ regularization (Equation  \ref{genReg}) respectively from photon density measurements with $1\%$ additive relative Gaussian noise. It can be seen that both, the reconstructions with the TV, and the mixed TV and $\ell_1$, regularizations obtain relatively good reconstructions from the true parameters. However, in the presence of multiple inclusions (Figure \ref{incfour}) or more complex inclusions (Figure \ref{halfmoon} and \ref{inc_1c}) the mixed regularizations seems to outperform the TV regularization. Note that in figures G-H, which represent reconstructions with the IRGN method, it can be seen that it performs similarly to the statistical methods in simple cases with one inclusion, but worse in setting with more complex inclusions.

\begin{figure}
\centering
\includegraphics[scale=0.34]{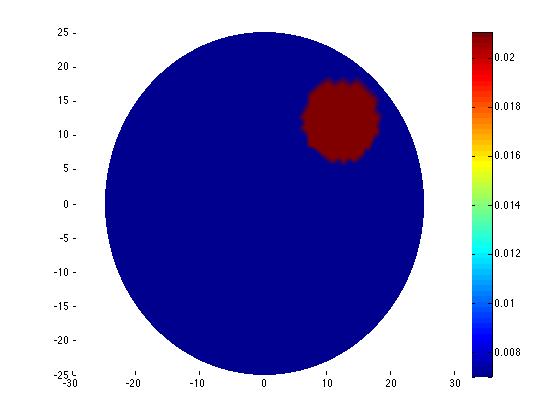}\includegraphics[scale=0.34]{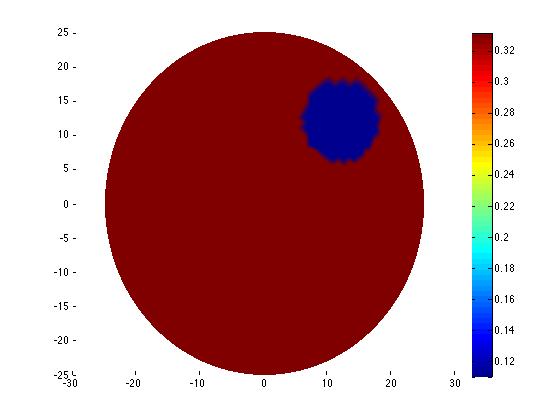}{\\\vspace*{-.5cm}\hspace*{-.4cm}\tiny A\hspace*{7.7cm}\tiny B}
\includegraphics[scale=0.34]{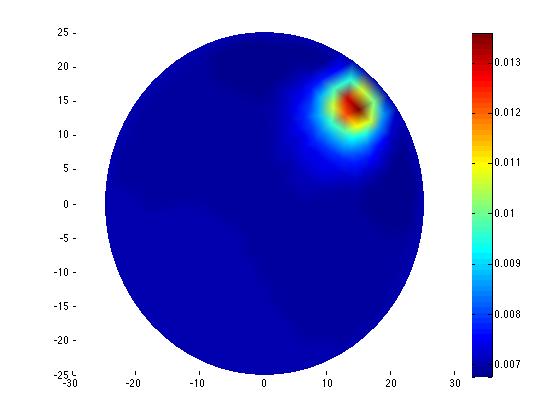}\includegraphics[scale=0.34]{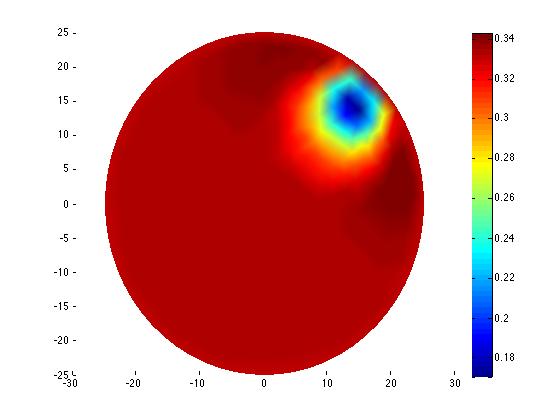}{\\\vspace*{-.5cm}\hspace*{-.4cm}\tiny C\hspace*{7.7cm}\tiny D}
\includegraphics[scale=0.34]{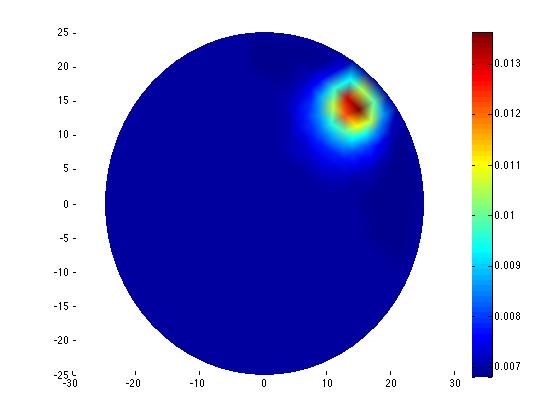}\includegraphics[scale=0.34]{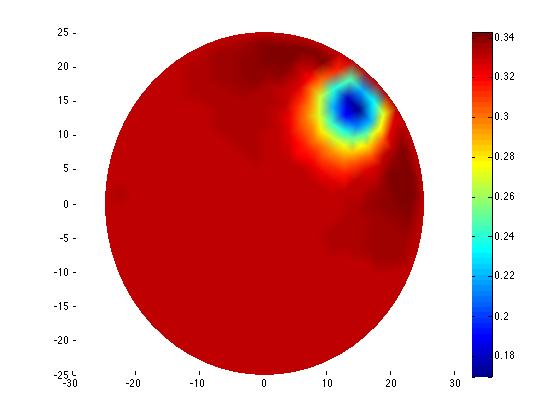}{\\\vspace*{-.5cm}\hspace*{-.4cm}\tiny E\hspace*{7.7cm}\tiny F}
\includegraphics[scale=0.34]{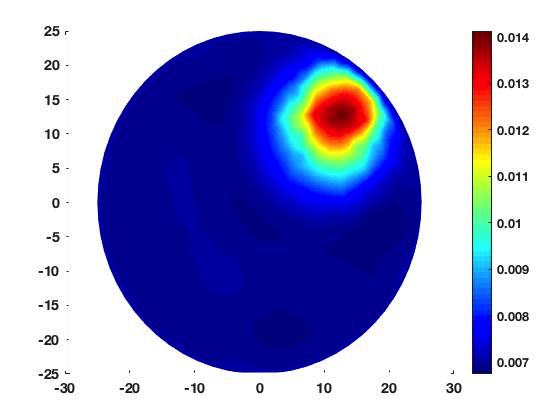}\includegraphics[scale=0.34]{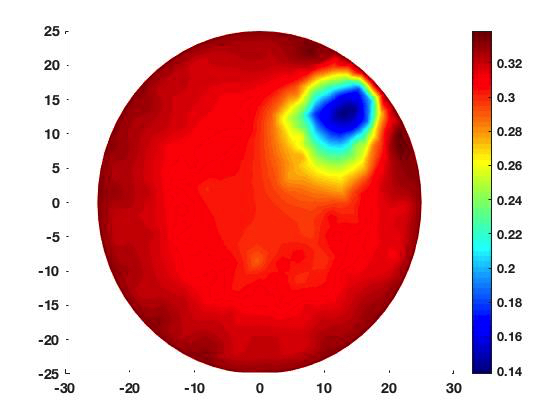}{\\\vspace*{-.5cm}\hspace*{-.4cm}\tiny G\hspace*{7.7cm}\tiny H}
\caption{A \& B: The true parameters $\mu$ and $D$ in the simulation mesh (Figure \ref{meshes}.A).  Reconstructions of the parameters $\mu$ and $D$ from measurements with $1\%$ additive Gaussian noise, C \& D: using the TV regularization (Section \ref{totalVar}).  E \& F: using a mixture of the TV and $\ell_1$ regularization (Equation  \ref{genReg}). G \& H: using IRGN Method.}\label{incone}
\end{figure}

\begin{figure}
\centering
\includegraphics[scale=0.34]{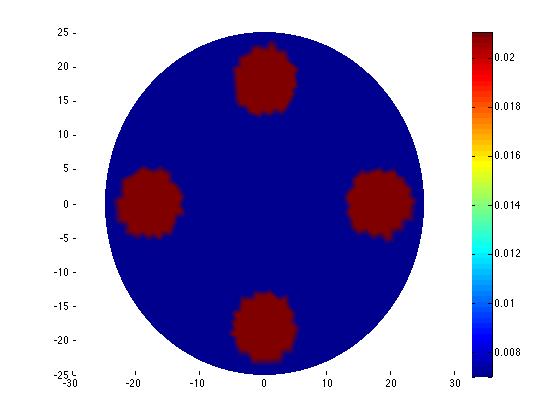}\includegraphics[scale=0.34]{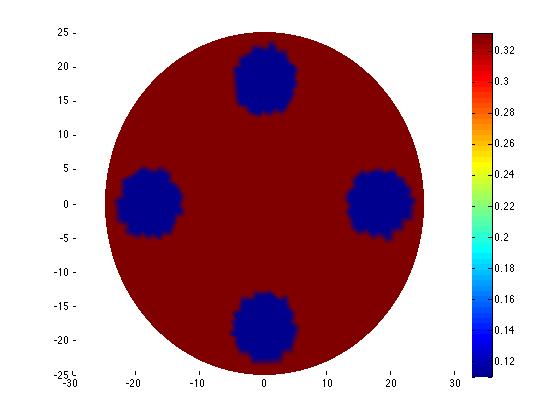}{\\\vspace*{-.5cm}\hspace*{-.4cm}\tiny A\hspace*{7.7cm}\tiny B}
\includegraphics[scale=0.34]{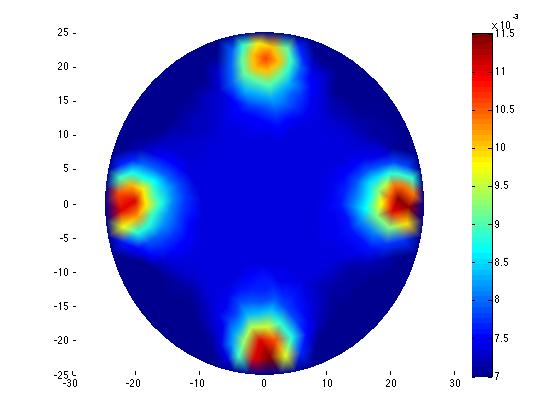}\includegraphics[scale=0.34]{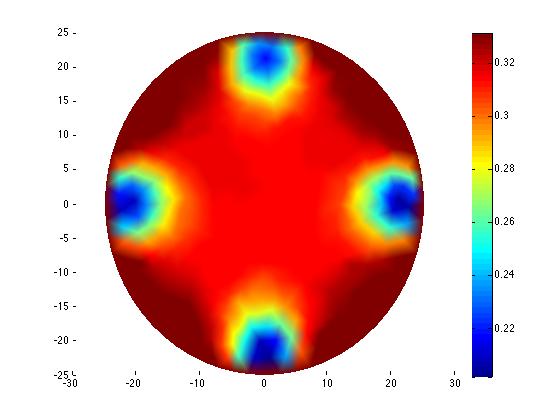}{\\\vspace*{-.5cm}\hspace*{-.4cm}\tiny C\hspace*{7.7cm}\tiny D}
\includegraphics[scale=0.34]{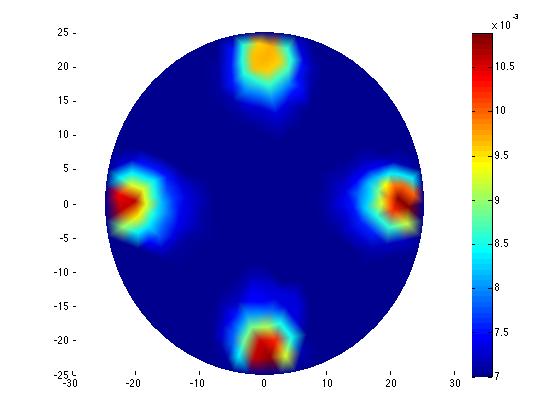}\includegraphics[scale=0.34]{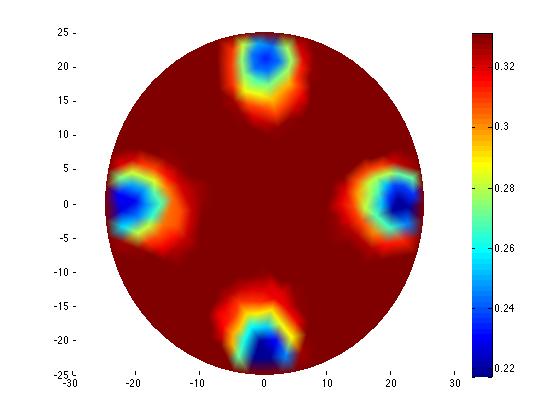}{\\\vspace*{-.5cm}\hspace*{-.4cm}\tiny E\hspace*{7.7cm}\tiny F}
\includegraphics[scale=0.34]{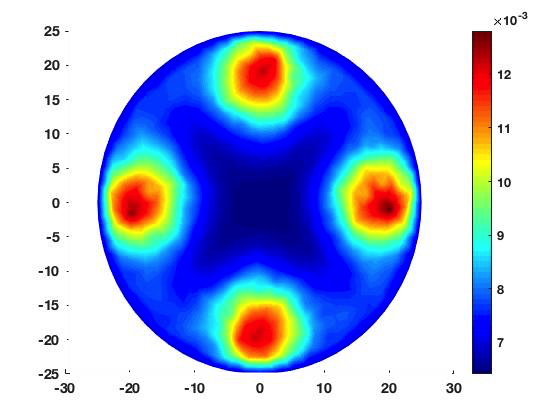}\includegraphics[scale=0.34]{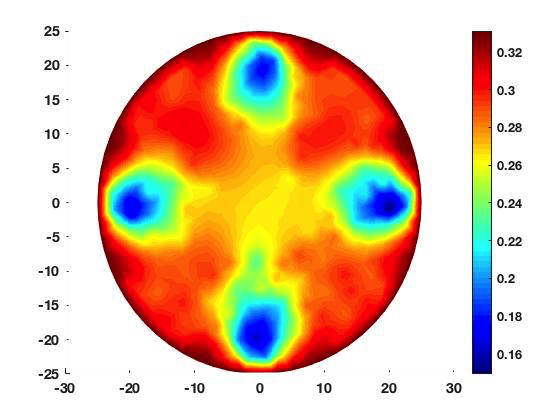}{\\\vspace*{-.5cm}\hspace*{-.4cm}\tiny G\hspace*{7.7cm}\tiny H}
\caption{A \& B: The true parameters $\mu$ and $D$ in the simulation mesh (Figure \ref{meshes}.A).  Reconstructions of the parameters $\mu$ and $D$ from measurements with $1\%$ additive Gaussian noise, C \& D: using the TV regularization (Section \ref{totalVar}).  E \& F: using a mixture of the TV and $\ell_1$ regularization (Equation  \ref{genReg}). G \& H: using IRGN Method.}\label{incfour}
\end{figure}
%%%%%%%%%%%%%%%%%%%%%%%%%%%%%%%%
\begin{figure}
\centering
\includegraphics[scale=0.34]{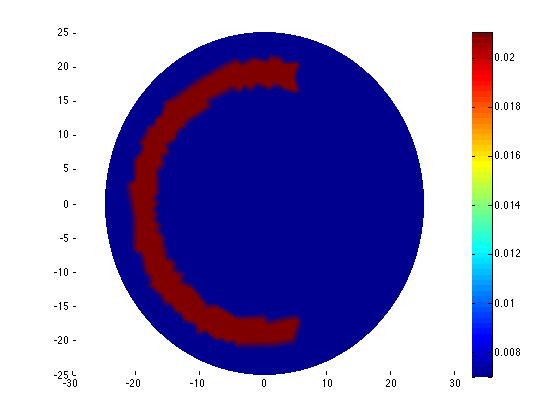}\includegraphics[scale=0.34]{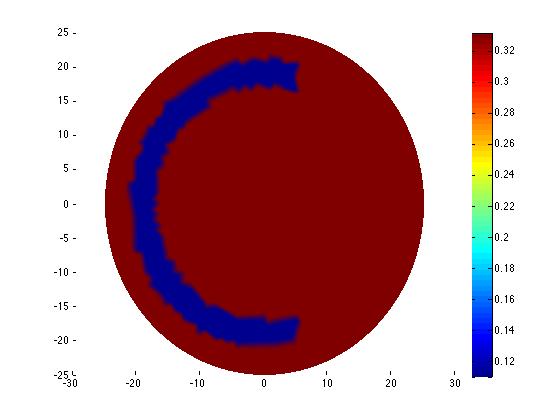}{\\\vspace*{-.5cm}\hspace*{-.4cm}\tiny A\hspace*{7.7cm}\tiny B}
\includegraphics[scale=0.34]{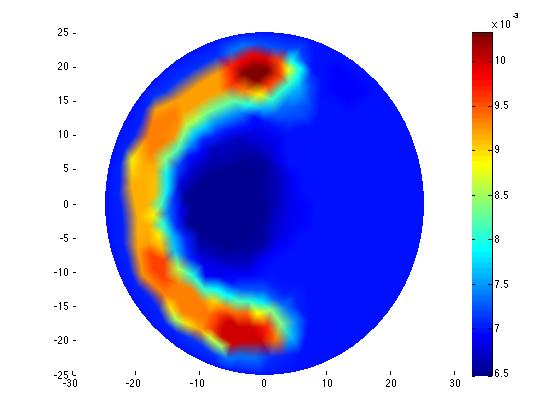}\includegraphics[scale=0.34]{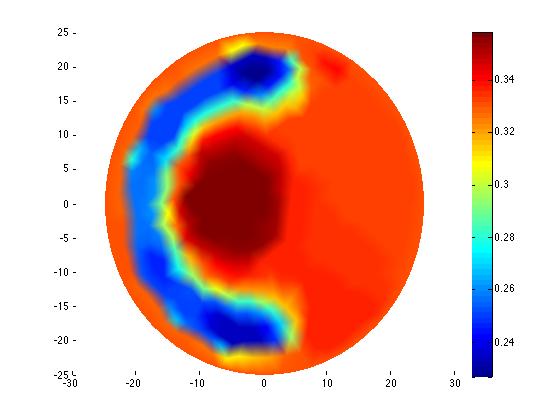}{\\\vspace*{-.5cm}\hspace*{-.4cm}\tiny C\hspace*{7.7cm}\tiny D}
\includegraphics[scale=0.34]{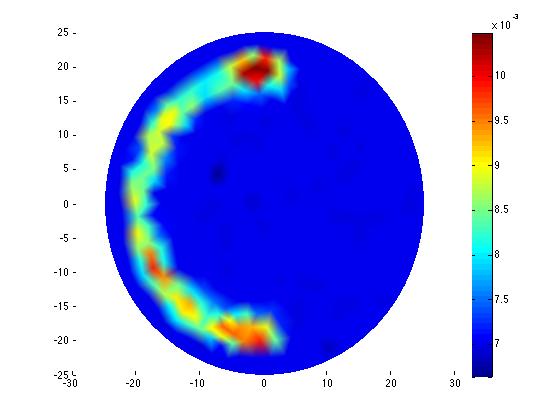}\includegraphics[scale=0.34]{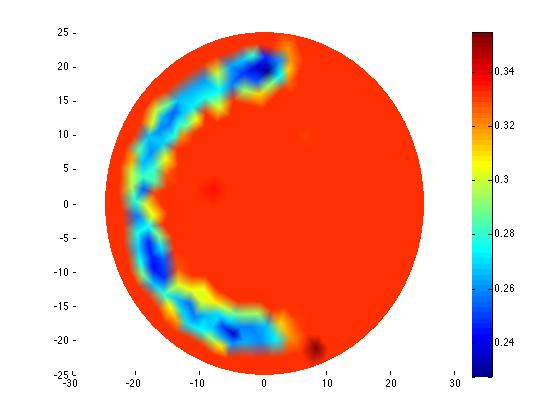}{\\\vspace*{-.5cm}\hspace*{-.4cm}\tiny E\hspace*{7.7cm}\tiny F}
\includegraphics[scale=0.34]{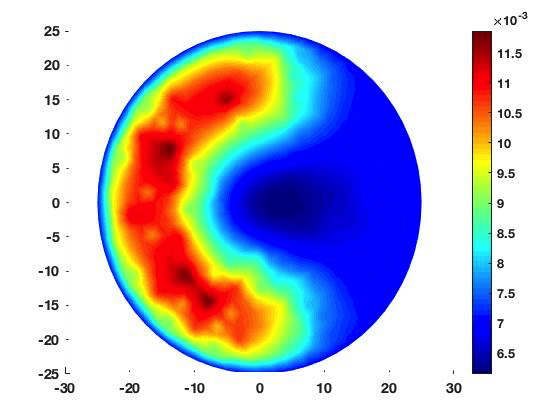}\includegraphics[scale=0.34]{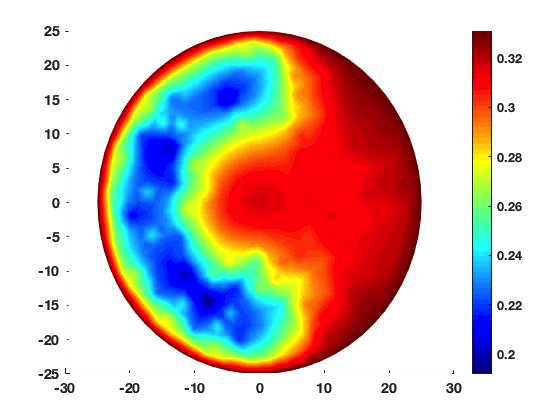}{\\\vspace*{-.5cm}\hspace*{-.4cm}\tiny G\hspace*{7.7cm}\tiny H}
\caption{A \& B: The true parameters $\mu$ and $D$ in the simulation mesh (Figure \ref{meshes}.A).  Reconstructions of the parameters $\mu$ and $D$ from measurements with $1\%$ additive Gaussian noise, C \& D: using the TV regularization (Section \ref{totalVar}).  E \& F: using a mixture of the TV and $\ell_1$ regularization (Equation  \ref{genReg}). G \& H: using IRGN Method.}\label{halfmoon}
\end{figure}
%%%%%%%%%%%%%%%%%%%%%%%%%%%
\begin{figure}
\centering
\includegraphics[scale=0.33]{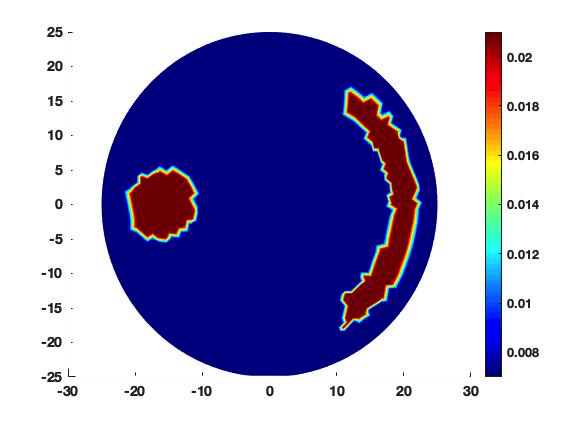}\includegraphics[scale=0.33]{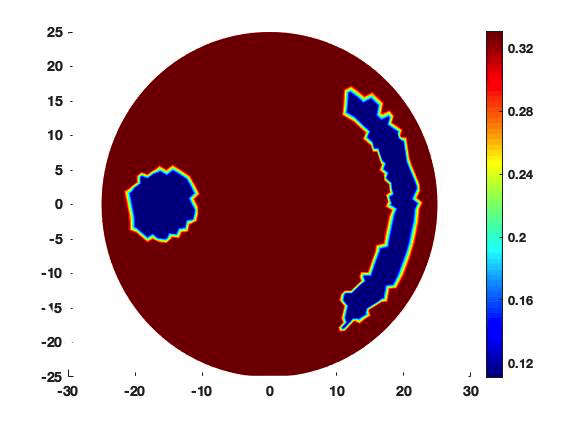}{\\\vspace*{-.5cm}\hspace*{-.4cm}\tiny A\hspace*{7.7cm}\tiny B}
\includegraphics[scale=0.34]{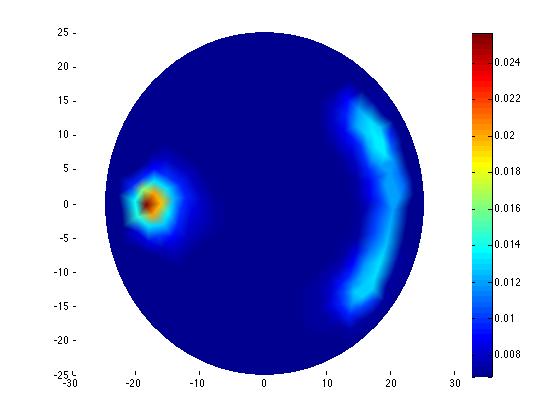}\includegraphics[scale=0.34]{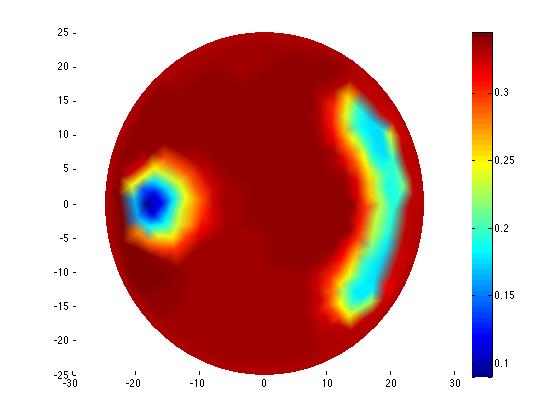}{\\\vspace*{-.5cm}\hspace*{-.4cm}\tiny C\hspace*{7.7cm}\tiny D}
\includegraphics[scale=0.34]{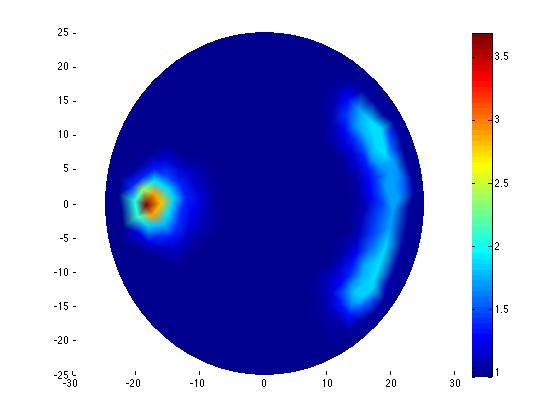}\includegraphics[scale=0.34]{geom1c_1n_D_stat.jpg}{\\\vspace*{-.5cm}\hspace*{-.4cm}\tiny E\hspace*{7.7cm}\tiny F}
\includegraphics[scale=0.33]{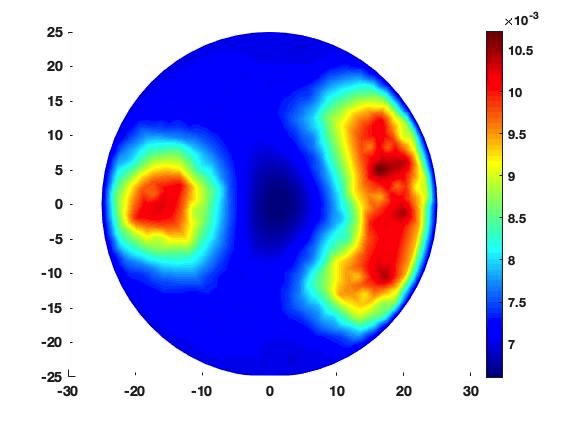}\includegraphics[scale=0.33]{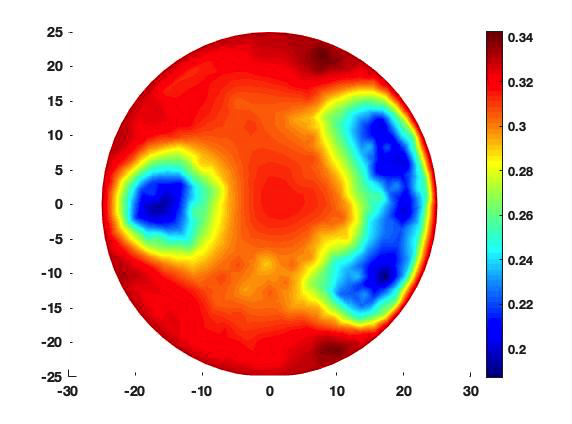}{\\\vspace*{-.5cm}\hspace*{-.4cm}\tiny G\hspace*{7.7cm}\tiny H}
\caption{A \& B: The true parameters $\mu$ and $D$ in the simulation mesh (Figure \ref{meshes}.A).  Reconstructions of the parameters $\mu$ and $D$ from measurements with $1\%$ additive Gaussian noise, C \& D: using the TV regularization (Section \ref{totalVar}).  E \& F: using a mixture of the TV and $\ell_1$ regularization (Equation  \ref{genReg}). G \& H: using IRGN Method.}\label{inc_1c}
\end{figure}
%%%%%%%%%%%%%%%%%%%%%%%%%%%%

Clearly, the reconstructions are strongly dependent on the choice of the regularization parameters. There is a vast literature for choosing optimal regularization parameters for linear problems. However, there are, to our knowledge, no good methods for nonlinear problems like DOT. Hence, we chose the parameters add hock. That is we used a computer cluster to run the algorithm with large set of regularization parameters choices, then we evaluated the reconstructions and picked the visually  best parameters for the TV and the mixed regularizations. Note that once this parameter was found it was kept fixed for all reconstructions in figures \ref{incone}-\ref{inc_1c}. 

\begin{table}[!htb]
    \centering
    \begin{tabular}{|c|c|c|c|c|}
    \hline
Example  & Example 1 & Example 2 & Example 3 & Example 4  \\ \hline
$\xi$  & 0.0018 & 0.0016 &0.0016 & 0.0019\\ \hline
Residual, $E_N$ & 0.0124 & 0.010115 & 0.01335 & 0.01302 \\ \hline
Residual, $E_S$ & 0.0129 & 0.0153 & 0.0139 & 0.0139\\ \hline
    \end{tabular}
    \caption{Numerical results for noiselevel $\xi$ and residual error (i) $E_N$ using IRGN and (ii) $E_S$ using Statistical inversion for 1$\%$ noise}    
    \label{tab:noise_res_1}
\end{table}
%%%%%%%%%%%%%%%%%%%%%%%%%%
% \begin{table}[!htb]
%     \centering
%     \begin{tabular}{|c|c|c|c|c|}
%      \hline
% Example  & Example 1 & Example 2 & Example 3 & Example 4  \\
% \hline
% $\xi$  & 0.0064 & 0.0071 &0.0057 & 0.0064 \\
% \hline
% Residual, $E_N$ & 0.011 & 0.01423 & 0.01516 & 0.01355 \\ \hline
% Residual, $E_S$ & 0.0134 & 0.0154 & 0.0142 & 0.0149\\ \hline
%     \end{tabular}
%     \caption{Numerical results for noiselevel $\xi$ and residual error (i) $E_N$ using IRGN and (ii) $E_S$ using Statistical inversion for 5$\%$ noise}
%     \label{tab:noise_res_5}
% \end{table}
%%%%%%%%%%%%%%%%%%%%%%%%%
 \begin{figure}[!htb]
     \centering
     \includegraphics[scale=0.27]{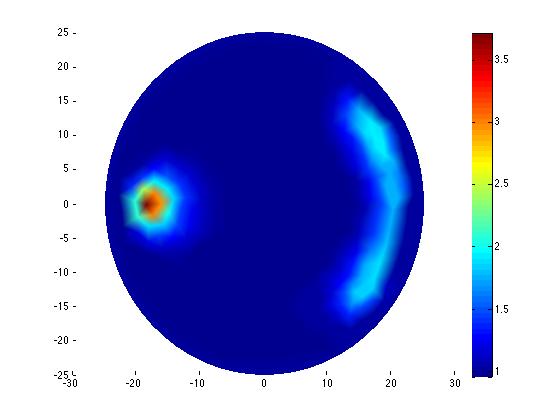}\includegraphics[scale=0.27]{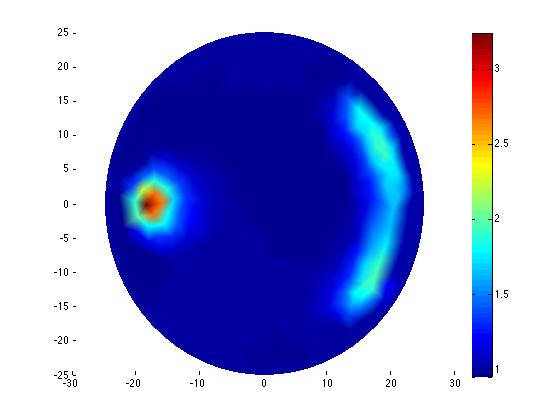}\includegraphics[scale=0.27]{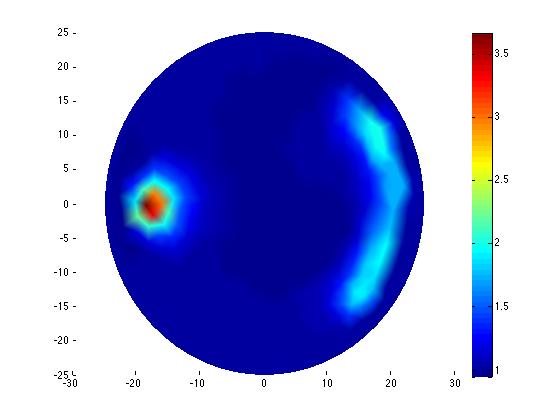}{\\\vspace*{-.5cm}\hspace*{-.4cm}\tiny A\hspace*{5.7cm}\tiny B\hspace*{4.7cm}\tiny C}
     
     \includegraphics[scale=0.27]{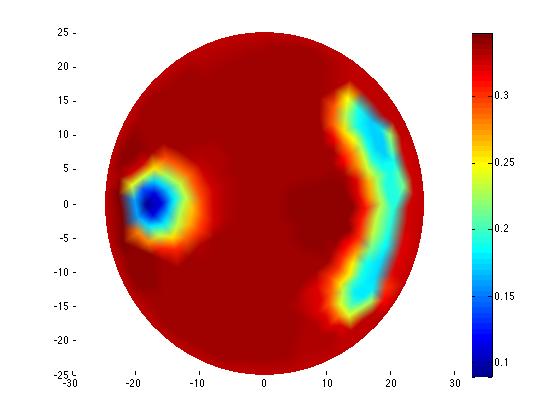}\includegraphics[scale=0.27]{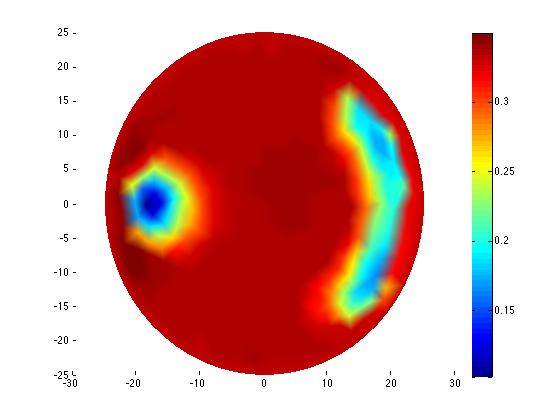}\includegraphics[scale=0.27]{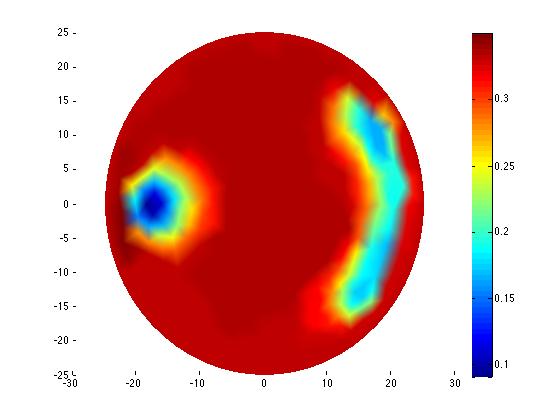}{\\\vspace*{-.5cm}\hspace*{-.4cm}\tiny D\hspace*{5.7cm}\tiny E\hspace*{4.7cm}\tiny F}
          \caption{Reconstruction of $\mu$ and $D$ using statistical inversion method with, A, D: 5$\%$ noise, B, E: with 10$\%$ noise, C, F: with 20$\%$ relative additive Gaussian noise, respectively.}
     \label{fig:absorption_noise}
 \end{figure}
 %%%%%%%%%%%%%%%%%%%%%%%%%
 \begin{figure}[!htb]
     \centering
     \includegraphics[scale=0.27]{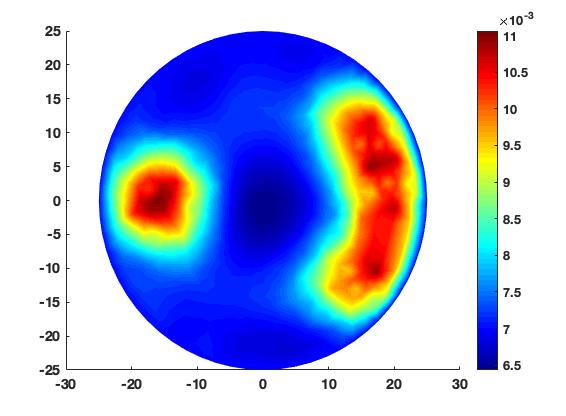} \includegraphics[scale=0.27]{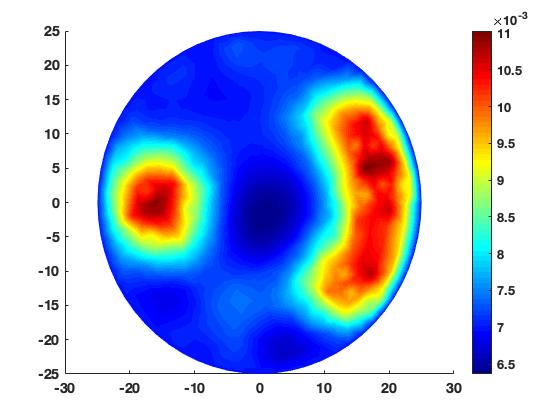}\includegraphics[scale=0.27]{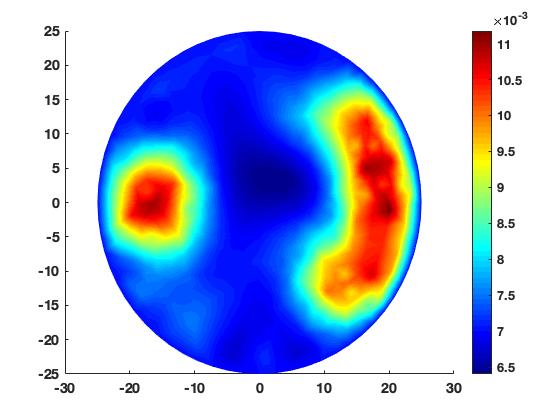}{\\\vspace*{-.5cm}\hspace*{-.4cm}\tiny A\hspace*{5.7cm}\tiny B\hspace*{4.7cm}\tiny C}
     
     \includegraphics[scale=0.27]{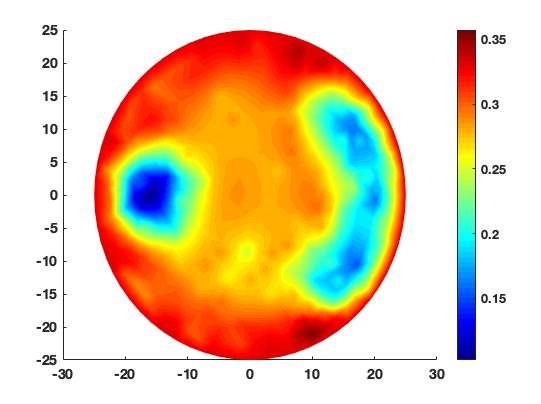}\includegraphics[scale=0.27]{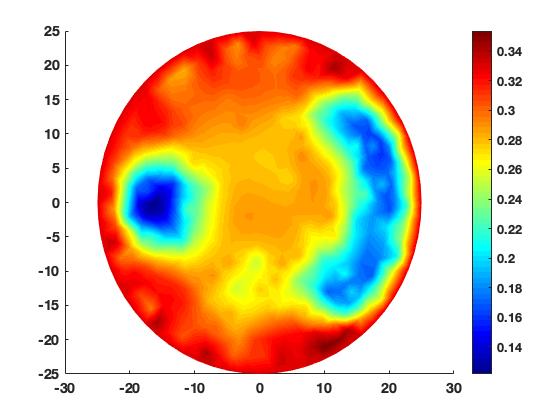}\includegraphics[scale=0.27]{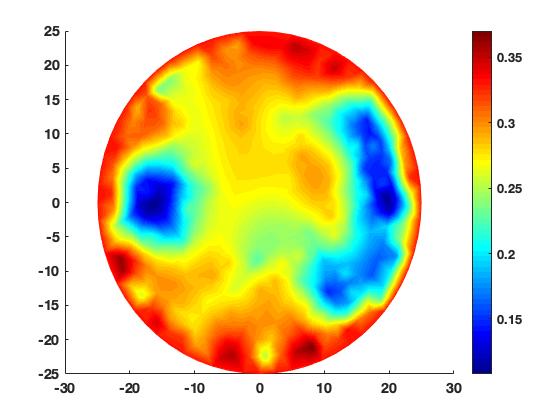}{\\\vspace*{-.5cm}\hspace*{-.4cm}\tiny D\hspace*{5.7cm}\tiny E\hspace*{4.7cm}\tiny F}
          \caption{Reconstruction of $\mu$ and $D$ using IRGN method with, A, D: 5$\%$ noise, B, E: with 10$\%$ noise, C, F: with 20$\%$ relative additive Gaussian noise, respectively.}
     \label{fig:irgn_noise}
 \end{figure}
%   \begin{figure}[!htb]
%      \centering
%      \includegraphics[scale=0.27]{geom1c_5n_D_mixTV.jpg}\includegraphics[scale=0.27]{geom1c_10n_D_mixTV.jpg}\includegraphics[scale=0.27]{geom1c_20n_D_mixTV.jpg}{\\\vspace*{-.5cm}\hspace*{-.4cm}\tiny A\hspace*{5.7cm}\tiny B\hspace*{4.7cm}\tiny C}

%           \caption{Reconstruction of $D$ with A: 5$\%$ noise, B: with 10$\%$ noise, C: with 20$\%$ additive Gaussian noise}
%      \label{fig:diffusion_noise}
%  \end{figure}
 %%%%%%%%%%%%%%%%%%%%
\begin{table}[!htb]
    \centering
    \begin{tabular}{ |p{1.9cm}||p{1.6cm}|p{1.6cm}|p{1.6cm}|p{1.6cm}|p{1.6cm}|p{1.6cm}|}
 \hline 
   Relative  & $L_1$ error & $L_1$ error & $L_1$ error & $L_2$ error  & $L_2$ error & $L_2$ error \\
   noise level & (TV) & (GR) &  (IRGN) &  (TV) &  (GR) & (IRGN) \\
%   Relative noise level & $L_1$ error (TV)& $L_1$ error (GR) & $L_1$ error (IRGN) & $L_2$ error (TV) & $L_2$ error (GR) & $L_2$ error (IRGN) 
 \hline 
1\%   & 0.1447    & 0.1337 & 0.1234 &   0.3245 &0.3398 & 0.3230\\
 5\% &   0.1439 & 0.1347  & 0.1206 & 0.3192&0.3423 & 0.3212\\
 10\%   & 0.1462    &0.1355 & 0.1258  & 0.3272&0.3390 & 0.3257\\
 15\% &   0.1442  & 0.1328 & 0.1315 &  0.3232&0.3368 & 0.3287\\
 20\%   & 0.1404    &0.1358 & 0.1144 &  0.3129&0.3441 & 0.3268\\
%  25\% &   0.1424  & 0.1473 & div &  0.3293&0.3734 & div \\ 
% 30\% &   0.1554  & 0.1730 & div &  0.3423&0.4232 & div \\
 \hline
    \end{tabular}
    \caption{Relative Numerical Errors of $\mu$}
    \label{tab:error_mu}
\end{table}

\begin{table}[!htb]
    \centering
    \begin{tabular}{ |p{1.9cm}||p{1.6cm}|p{1.6cm}|p{1.6cm}|p{1.6cm}|p{1.6cm}|p{1.6cm}|}
 \hline
   Relative  & $L_1$ error & $L_1$ error & $L_1$ error & $L_2$ error  & $L_2$ error & $L_2$ error \\
   noise level & (TV) & (GR) &  (IRGN) &  (TV) &  (GR) & (IRGN) \\
 \hline
1\%   & 0.0625    & 0.0470 & 0.1002 & 0.1189 &0.1227 & 0.1259\\
 5\% &  0.0625 & 0.0455  & 0.0930 & 0.1161 & 0.1198 & 0.1192\\
 10\%   & 0.0638 &  0.0462 & 0.1159 & 0.1199 & 0.1216 & 0.1423 \\
 15\% & 0.0632 & 0.0449 & 0.1363 & 0.1204 & 0.1203 & 0.1653\\
  20\%   & 0.0598  & 0.0474 & 0.0923 & 0.1158 & 0.1242 & 0.1171\\
%  25\% &   0.0608  & 0.0557 & div & 0.1279 & 0.1461 & div \\ 
% 30\% &  0.0753 & 0.0751 & div & 0.1442 & 0.2055 & div \\
 \hline
    \end{tabular}
    \caption{Relative Numerical Errors of $D$}
    \label{tab:error_D}
\end{table}

In Table \ref{tab:error_mu} and \ref{tab:error_D} the relative error of reconstructions of $\mu$ and $D$ with different relative noise levels where computed. Note that the relative error of $\mu$ is defined as $\frac{||\mu_{t}-\mu_{r}||_{L_p}}{||\mu_{t}||_{L_p}}$, were $\mu_{t}$ represents the true parameter to be estimated and $\mu_{r}$ the reconstruction of $\mu_{t}$. The relative error of $D$ is defined analogous. In table \ref{tab:error_mu} and \ref{tab:error_D} the relative $L_1$ and $L_2$ errors from the reconstructions using the Total Variation regularization (TV) defined in section \ref{totalVar}, the General Regularization (GR) defined in (\ref{genReg}) as well as with the IRGN method have been computed. As expected the reconstructions with using the Total Variation have mostly a smaller $L_2$ relative errors while reconstructions using the General Regularization have mostly a smaller $L_1$ relative errors. Furthermore, it can be seen that the presented method obtained relative good reconstructions up to a relative noise level from about $20\%$ (see figure \ref{fig:absorption_noise}). 
Note that the IRGN method generally seems to under-perform when comparing its relative error  with the statistical methods. Moreover, the statistical inversion method and the IRGN method become unstable for higher noise level.

%%%%%%% END NEW ERROR TABLE %%%%%%%%%%%%%%%%%%%%%%%%%%%%%%%%%%%%%%%%%

\section{Conclusion}\label{conclusion}

We presented a statistical formulation for the DOT inverse problem and compared it with the classical IRGN method. We found that from the visual point of view the statistical method outperforms the IRGN method up to high levels of relative noise. However, due to using the MCMC method the statistical solution is  computational  more intensive than the IRGN method. 

The main contributions of this paper can be summarized in:
\begin{itemize}
\item To the best of our knowledge this is the first paper solving the Statistical Inverse problem for DOT.
\item We introduce a mixture of the $TV$ and the $\ell_p$ regularization for the statistical DOT paper.
\item We present many numerical experiments to evaluate the performance of the statistical inverse problem and compare it with the IRGN method as baseline.
\end{itemize}

% \section{Acknowledgement} We thank the University of Brement and Humboldt Foundation to provide support, Clemson for exchange program which resulted in the connection with Bremen and particulary we had the pleausre of meeting Armin Lechlieter, etc. Thilo was also interacted as a graduate student./

\newpage
\bibliographystyle{plain}
\bibliography{references}

\end{document}